\documentclass[UTF-8,11pt,a4paper]{article}
\usepackage{thisreport}
\usepackage{listings}
\usepackage{cases}
\usepackage{soul}
\usepackage{color,xcolor}
\usepackage{graphicx}
\usepackage{authblk} 

\usepackage[numbers,sort,compress]{natbib}
\graphicspath{{./figures/}}
\usepackage{subfigure}
\usepackage{algorithm}
\usepackage{algorithmic}
\usepackage{hyperref}

\def\stuNum{}
\newcommand\keywords[1]{\textbf{Keywords}: #1}

\title{Learning based numerical methods for Helmholtz equation with high frequency}

\author[1]{Yu Chen}
\author[2]{Jin Cheng \thanks{Corresponding author: jcheng@fudan.edu.cn}}
\author[2]{Tingyue Li}
\author[3]{Yun Miao}

\affil[1]{School of Mathematics, Shanghai University of Finance and Economics, Shanghai 200433, P. R. China}
\affil[2]{School of Mathematical Sciences, Fudan University, Shanghai 200433, P. R. China}
\affil[3]{Theory lab, Central Research institute 2012 Labs, Huawei Technologies Co ltd. , Shanghai 200433, P. R. China}

\date{}

\begin{document}
\pagenumbering{arabic}

\maketitle
\begin{abstract}
	High-frequency issues have been remarkably challenges in numerical methods for partial differential equations. In this paper, a learning based numerical method (LbNM) is proposed for Helmholtz equation with high frequency. The main novelty is using Tikhonov regularization method to stably learn the solution operator by utilizing relevant information especially the fundamental solutions. Then applying the solution operator to a new boundary input could quickly update the solution. Based on the method of fundamental solutions and the quantitative Runge approximation, we give the error estimate. This indicates interpretability and generalizability of the present method. Numerical results validates the error analysis and demonstrates the high-precision and high-efficiency features. 
\end{abstract}


\section{Introduction}
Numerical methods for partial differential equations are powerful tools for solving the engineering problems. The related research has developed rapidly in the last century. Some famous commercial software such as Matlab, Comsol, etc, have been widely used by the engineers and researchers. Even though, there are still some difficult problems, which may not be solved successfully, for example, the high frequency problems, high dimensional problems etc. Artificial intelligence is a rapidly developed technology and will give more influence on the daily life and traditional researches. It would be interesting to combine the AI technology to the numerical algorithms and simulation of the practical problems. The goal of this paper is to apply the learning theory in the numerical solutions of partial differential equations. We will present a learning based numerical method for partial differential equations. It is shown that, compared with other methods,  our method has the strong interpretability and generalizability. Some error analysis can be proved and confirmed with the numerical simulations.

Helmholtz equations are fundamental equations in the scattering theory and have been studied for a long time. 
There are many researches on the increasing stability for Helmholtz equation \cite{Isakov2004,Nagayasu2013,Cheng2016,Isakov2020}. To the best of the authors' knowledge, it is hard to give numerical examples to verify the theoretical estimates of increasing stability owning to the difficulty of computing the Helmholtz equation with high frequencies. However, our learning based numerical method is effective in the performance of high-frequency problems.

The method of fundamental solutions (MFS) is useful to finish our error analysis in this paper. It is well-known that it has been widely used in solving numerically boundary value problems for linear partial differential equations since Mathon and Johnston formulated the method numerically in \cite{Mathon1977} for the first time. The authors give an overview of the MFS as a heuristic numerical method in \cite{Alexander2020}. In short, it is a mesh-free method that the solution is approximated by a linear combination of the fundamental solutions of the partial differential equations with respect to some source points placed outside the domain.
However, the main drawback of the MFS is that the matrices involved are typically
ill-conditioned. The authors proposed techniques to alleviate or remove the ill-conditioning of the MFS, but none of these methods seems to completely solve the problem in 
\cite{Antunes2018_1,Antunes2018_2,Antunes2022}. Barnett and Betcke \cite{Barnett2008} investigated the MFS for the interior Helmholtz problem on analytical domains to obtain a numerically stable method.
Besides, facing new problems, we need to  compute the solution repeatedly by solving the linear system of equations.

One property of elliptic equations that plays key role in our error analysis is Runge's approximation properties for elliptic equations. The classical Runge approximation is studied by Lax \cite{Lax1956} and Malgrage \cite{Malgrange1956}.
It states that the solution of the partial differential equation in the smaller domain can be approximated by the solution of the same equation in the larger domain.
The quantitative estimate of Runge approximation is first considered in \cite{Salo2019} and further gives the bound of the approximated solution on the larger domain. It has become an increasingly significant tool in inverse problems \cite{Ruland2021,Pohjola2022,Kravchenko2021}. 
The numerical realization of Runge's approximation was considered in \cite{Cheng2005}. 

In this paper, we propose a learning based numerical method for Helmholtz equation with high frequency. The methods are high-precision and high-efficiency as they utilize various data solution, especially the fundamental solutions. Having learned the solution operator, we can update the solution of new problems with different boundary conditions. In addition, it is practical to handle the case of computing local areas in engineering problems. 

The main idea of the theoretical analysis is to make use of the theoretical results of the MFS and Runge approximation. Then, we divide into three cases based on the region that the solution can be continued.
In Case 1, we employ Tikhonov regularization and the theoretical results of MFS to derive the Lipschitz-type error estimate when the solution can be continued to a sufficiently large domain. Furthermore, using the quantitative version of Runge approximation in \cite{Salo2019}, we can give the H\"{o}lder-type estimate of error precisely in case 2 of Section \ref{sec:case2} although the solution can only be continued to a slightly large domain. In case 3 of Section \ref{sec:case3}, we can obtain a logarithm-type error estimate 
even the solution can not be continued at all. 

This paper is organized as follows. The learning based numerical methods will be given in Section \ref{sec:algorithm}. The corresponding lemmas and propositions of MFS and Runge approximation will be present in Section \ref{sec:preliminaries}.
The main theoretical analysis will be proved in section
\ref{sec:theoretical results}. In section \ref{sec:numerical examples}, numerical evidence shows that the learning based numerical method is a fast and high-precision approach, in particular in high-frequency regimes. Some concluding remarks and the future focus are given in section \ref{sec:concluding}.

\section{The learning based algorithm}\label{sec:algorithm}
\subsection{Motivation}
Suppose that $u(x)$ satisfies Helmholtz equation with Dirichlet boundary condition:
\begin{numcases}{}
	\Delta u(x)+k^2 u(x)=0, & \text{$x \in \Omega$}, \label{eq:Helm}\\
	u(x)=f(x), & \text{$x \in \partial\Omega $},\label{eq:Helm_boundary}
\end{numcases} 
where $k>0$ is the wave number. For a fixed point $x\in\Omega$, consider Green's function satisfying
\begin{equation} \label{eq:green}
	\begin{cases}
		\Delta G(x, y)+k^2 G(x, y)=\delta(y-x), &\text{$y \in \Omega$}, \\
		G(x, y)=0, &\text{$y \in \partial\Omega $}\,.
	\end{cases}
\end{equation}
According to Green's formula,
the solution can be formulated as 
$$
u(x)=\int_{\partial \Omega} \frac{\partial G(x, y)}{\partial n} f(y) \mathrm{d} S_y,\quad x\in \Omega,
$$
where $n$ is the outward normal vector.
Define the following bounded linear operator 
$$
\begin{aligned}
	A: H^{\frac{1}{2}}(\partial \Omega) &\rightarrow H^{1}(\Omega)\\
	Af& = \int_{\partial \Omega} \frac{\partial G(x, y)}{\partial n} f(y) \mathrm{d} S_y.
\end{aligned}
$$
If the boundary condition is discretized by 
$$
f\approx\sum\limits_{i=1}^{N} f_i e_i(x)
$$
with $\left\{e_i(x)\right\}_{i=1}^{N}$ being interpolation basis functions on the boundary, then the solution has the form
$$
u(x) \approx \sum_{i=1}^N f_i \int_{\partial \Omega} \frac{\partial G}{\partial n}(x, y) e_i(y) \mathrm{d} y=: \sum_i A_i(x) f_i,
$$
where $A_{i}(x)=(Ae_{i})(x)$ are components of the discrete solution operator.
\subsection{Algorithm}
In this work, we consider a novel method based on reconstruction of the solution operator from sample solutions. Assuming there are plenty of solutions with the corresponding boundary values, we can reconstruct the solution operator by these solution data. We will show that with suitably chosen solutions, it
is possible to obtain reliable and accurate solution operator. 
Taking the boundary value problem \eqref{eq:Helm}-\eqref{eq:Helm_boundary} as example, let $u_i(x),\ (i=1,2, \cdots, M)$ be a set of solutions to Helmholtz equation with different boundary values (training solutions), and denote $T=\{x_j\}_{j=1}^N \subset\partial\Omega$ be the set of collocation points (interpolation nodes) on the boundary such that $f(x)\approx \sum_{i=1}^N f(x_i)e_i(x)$ on $\partial \Omega$. Then generally for each sample solution $u_i(x)$, $x\in\Omega$, there exists the following approximation, 
$$
u_i(x) \approx \sum_{j=1}^N A_j\left(x\right) u_i\left(x_j\right),\quad x_j\in T
$$ 
and there forms the following linear system,
$$
\left(u_1(x), \cdots, u_M(x)\right)=\left(A_1(x), \cdots, A_N(x)\right)\left(\begin{array}{ccc}
	u_1\left(x_1\right) & \cdots & u_M\left(x_1\right) \\
	\vdots & \ddots & \vdots \\
	u_1\left(x_N\right) & \cdots & u_M\left(x_N\right)
\end{array}\right)\,,
$$
denoted by, 
$$
\bm{b}_{1 \times M}=\bm{a}_{1 \times N} V_{N \times M}.
$$
When $M>N$, it is an over-determined system. 
We obtain the solution solving the Tikhonov minimization problem,
\begin{align}\label{eq:Tikhonov}
	\bm{a}_{*}=\underset{\bm{a} \in \mathbb{R}^N}{\arg \min } \left\{\|\bm{a}  V-\bm{b}\|^{2}+\alpha \|\bm{a}\|^2\right\}
\end{align}
in which $\alpha\|\bm{a}\|^2$ is the penalty term with an {\it a priori} given regularization parameter $\alpha>0$. Referring to the conclusions of Tikhonov regularization in \cite{Tikhonov1977}, \eqref{eq:Tikhonov} exists a unique minimum $\bm{a}_{*}$ for all $\alpha>0$, this minimum is given by 
\begin{align}\label{eq:minimum}
	\bm{a}_{*} = \bm{b} V^*(VV^*+\alpha I)^{-1}=:\bm{b}\bm{W},
\end{align}
which depends continuously on $\bm{b}^{(x)} $. Then, the solution under boundary input $f$ can be calculated by,
$$
u(x)\approx \bm{a}_*\cdot \bm{f}= \bm{b} V^*(VV^*+\alpha I)^{-1}\bm{f}\,,
$$
in which $\bm{f}=(f(x_{1}),f(x_{2}),\cdots,f(x_{N}))^{T}$.

Particularly, the training solutions $u_i(x)$ can be taken as fundamental solutions to Helmholtz equation
$$
\Phi\left(\hat{x}_i, \cdot\right)= 
\begin{cases}
	\displaystyle\frac{\mathrm{i}}{4} H_0^{(1)}(k|\hat{x}_i-\cdot|), & n=2\,,\\
	\displaystyle\frac{1}{4 \pi} \frac{\exp (\mathrm{i} k|\hat{x}_i-\cdot|)}{|\hat{x}_i-\cdot|}, & n=3\,,
	
\end{cases}
$$
where $\left\{\hat{x}_i\right\}_{i=1}^M$ are source points equally distributed on $\partial O_{R}\left(\Omega \subset\subset O_{R}\right)$ (see Fig. \ref{fig:domains} for illustration), and  $H_0^{(1)}$ denotes the Hankel function of the first kind of order 0.
\begin{remark}
	For each fixed $x\in\Omega$, let $G_0(x,y)$ be the solution to 
	\begin{equation*} 
		\begin{cases}
			\Delta G_0(x, y)+k^2 G_0(x, y)=0, &\text{$y \in \Omega$}, \\
			G_0(x, y)=\Phi(x,y), &\text{$y \in \partial\Omega $}\,,
		\end{cases}
	\end{equation*}
	then Green's function defined in \eqref{eq:green} reads,  
	\begin{align*}
		G(x,y)=G_{0}(x,y)-\Phi(x,y).
	\end{align*}
	
\end{remark}
\begin{remark}
	The algorithm is also applicable to other types of boundary conditions by adjusting the data matrix correspondingly, which is studied in our forthcoming paper.

\end{remark}

\section{Preliminaries}\label{sec:preliminaries}

In this section, we introduce some preliminaries, including the global estimate and the interior regularity for Helmholtz equation, the method of fundamental solutions, and Runge approximation. These knowledge will be utilized in the following theoretical analysis.

The following lemmas provide the existence, the global estimate and the interior regularity for solutions of
the homogeneous Helmholtz equation, see \cite[Proposition 2]{Beretta2016}, \cite[Lemma 1.4]{Ruland2021} and \cite[Section 6.3]{Evans2016} for details. 

\begin{lemma}(\cite[Proposition 2]{Beretta2016} and \cite[Lemma 1.4]{Ruland2021})\label{lem:existence}
	Let $\Omega$ be a bounded Lipschitz domain in $\mathbb{R}^n, f \in H^{1 / 2}(\partial \Omega)$. Then, there exists a discrete set  $E \subset \mathbb{R}$ such that, for every $k^2 \in\mathbb{R} \backslash E$, there exists a unique solution $u \in H^1(\Omega)$ of \eqref{eq:Helm}-\eqref{eq:Helm_boundary}.
	Furthermore, there exists a positive constant $C$ such that
	$$
	\|u\|_{H^1(\Omega)} \leqslant C\left(1+\frac{k^3}{\operatorname{dist}(k^2,E)}\right)\|f\|_{H^{1 / 2}(\partial \Omega)},
	$$
	where $C$ depends only on $\Omega$.
	More precisely, $E$ denotes the set of the Dirichlet eigenvalues of the operator $(-\Delta)$ on $\Omega$.
\end{lemma}
\begin{lemma}\cite[Section 6.3]{Evans2016}\label{lem_interior estimate}
	Suppose that $u \in$ $H^{1}(\Omega)$ satisfies \eqref{eq:Helm} almost everywhere. Then for any $\Omega^{\prime} \subset\subset \Omega$,
	$$
	\|u\|_{H^{m+2, 2}\left(\Omega^{\prime}\right)} \leqslant C\|u\|_{L^2(\Omega)}
	$$
	for any $m\in \mathbb{N}$, 
	where the constant $C$ depends on $m$, the dimension $n$, the wave number $k$, and $\operatorname{dist}\left\{\Omega^{\prime}, \partial \Omega\right\}$.
\end{lemma}

The idea of the method of fundamental solutions (MFS) was initially proposed by Mathon and Johnston formulated the method numerically in \cite{Mathon1977}, where the solution is approximated by a linear combination of fundamental solutions.
Barnett and Betcke \cite{Barnett2008} gave the choices of charge points for the interior Helmholtz problem on analytical domains and the convergence rate bounds. They proposed that
\begin{equation}\label{eq:approx}
	u(x) \approx u^{(\mu)}(x)=\sum_{i=1}^M \mu_i \Phi(\hat{x}_{i},x) , \quad \hat{x}_{i} \in \partial O_{R}\,,
\end{equation}
where $\hat{x}_{i}$ are charge points equally spaced on a larger circle of radius $R>1$ (see Fig. \ref{fig:barnett}).
\begin{figure}[h]
	\centering  
	\includegraphics[width=0.5\textwidth]{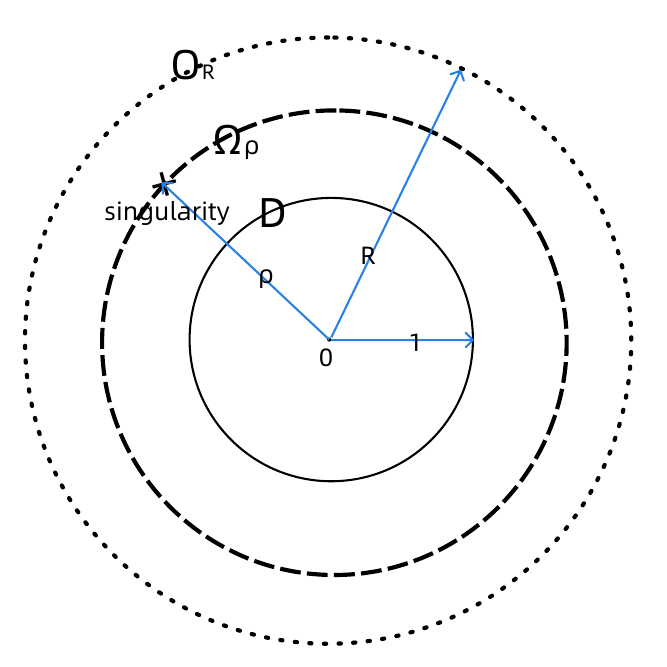}
	\caption{Geometry for the MFS in the unit disk.}
	\label{fig:barnett}
\end{figure}

Assume that $f$, defined on the unit circle, can be analytically continued to the annulus $\left\{z \in \mathbb{C}: \frac{1}{\rho}<|z|<\rho\right\}$ for some $\rho>1$, that is, the closest singularity of the analytic continuation of $f$ has the radius $\rho$ or $1 / \rho$. We then have asymptotically exponential decay of the Fourier coefficients,
\begin{equation}\label{eq:f_hat}
	|\hat{f}_{m}| \sim C \rho^{-|m|}, \quad|m| \rightarrow \infty
\end{equation}
for some constant $C$, where $\hat{f}_m$ denotes the $m$-th discrete Fourier coefficient of $f$.

Under the above assumption, there exists the following error estimate:
\begin{lemma}\label{FMS}
	Let $u$ defined on a unit disc $D$, assume that the boundary value $f$ can be analytically continued to be annulus $\{z\in \mathbb{C}:\frac{1}{\rho}<|z|<\rho\}$ for some $\rho>1$, i.e., the closest singularity of the analytic continuation of $f$ has the radius $\rho$ or $1/\rho$. Let $R>1$ and $M$ be even, then, for analytic boundary data $f$ obeying \eqref{eq:f_hat}, there exists MFS approximation $u^{(\mu)}$ in \eqref{eq:approx} with the boundary error satisfies
	\begin{equation}\label{eq:MFS_error}
		\|u-u^{(\mu)}\|_{L^2(\partial D)}\leqslant 
		\left\{
		\begin{array}{l}
			C\rho^{-M/2},\quad \rho < R^2,\\
			C\sqrt{M}R^{-M},\quad \rho = R^2,\\
			CR^{-M},\quad \rho >R^2,
		\end{array}
		\right.
	\end{equation}
	where $C$ depends on frequency $k$, $R$ and $f$, but independent of $M$.
\end{lemma}
For the proof, see \cite[Theorem 6]{Barnett2008}.
Based on Lemma \ref{FMS}, we can obtain the estimate for the norm of the coefficient vectors $\bm{\mu}=:(\mu_1,\cdots,\mu_M)$ in \eqref{eq:approx}.
\begin{lemma}
	Let $u$ be defined on a unit disc $D$, and $R<\rho$, with fixed analytic boundary data $f$ satisfying \eqref{eq:f_hat}. 
	Then the vector $\bm{\mu}$ corresponds to the MFS approximation $u^{(\mu)}$ is bounded by
	\begin{equation}\label{eq:MFS_mu}
		\|\bm{\mu}\|^{2}\leqslant \frac{C}{M}
	\end{equation}
	with corresponding boundary error norm \eqref{eq:MFS_error} as in Lemma \ref{FMS}, where $C$ is a constant independent on $f$.
\end{lemma}
\begin{proof}
	Expanding the coefficient vector $\bm{\mu}$ with a discrete Fourier basis labeled by $-\frac{M}{2}<j \leqslant \frac{M}{2}$,
	$$
	\mu_l=\sum_{j=-M / 2+1}^{M / 2} \hat{\mu}_j \mathrm{e}^{\mathrm{i} j \phi_l}, \quad \hat{\mu}_j=\frac{1}{M} \sum_{l=1}^M \mu_l \mathrm{e}^{-\mathrm{i} j \phi_l},
	$$
	where $\phi_{l}=\frac{2\pi l}{M}$.
	The Theorem 6 in \cite{Barnett2008} stated that the component $\hat{\mu}_j$ satisfies
	$$
	\hat{\mu}_j \sim C \frac{|j|}{M}\left(\frac{R}{\rho}\right)^{|j|}, \quad-\frac{M}{2}<j \leqslant \frac{M}{2}, \quad j \neq 0\,.
	$$
	Thus, 
	\begin{equation*}
		\|\bm{\mu}\|^2=M\|\hat{\bm{\mu}}\|^2\leqslant CM\sum\limits_{j=-M/2+1}^{M/2} \frac{|j|^2}{M^{2}}\left(\frac{R}{\rho}\right)^{2|j|}
		\leqslant \frac{C}{M}
	\end{equation*}
	for some constant $C$.
\end{proof}

When $\|u\|\leqslant 1$ in $\Omega$, the constant $C$ in \eqref{eq:MFS_error} and \eqref{eq:MFS_mu} would not depend on the solution $u$ itself. On account of the linear dependence of $\bm{\mu}$ on the solution, the following proposition is valid.
\begin{proposition}\label{bound-mu}
	Suppose that the solution is uniquely continuable to a sufficiently larger domain $\Omega_{\rho}$ than $O_{R}$ such that $\rho >R^{2}$ (see Fig. \ref{fig:case1} for example), then
	\begin{equation}\label{eq:MFS_error_L2}
		\|u-u^{(\mu)}\|_{L^2(\partial D)}\leqslant CR^{-M}\|u\|_{L^{2}(\Omega_{\rho})},\quad \rho >R^2,
	\end{equation}
	where $C$ depends on $k$, $R$, $d$ and $\Omega$, but independent of $M$. Moreover,
	\begin{equation}\label{eq:estimate_mu}
		\|\bm{\mu}\|^2\leqslant \frac{C}{M}\|u\|_{L^{2}(\Omega_{\rho})}^{2}.
	\end{equation}
\end{proposition}
\begin{proof}
	Let $\tilde{u}(x)=\displaystyle\frac{u(x)}{\|u\|_{L^{2}(\Omega_{\rho})}}$, replacing $u$ in \eqref{eq:MFS_error} with it, the conclusions \eqref{eq:MFS_error_L2} and \eqref{eq:estimate_mu} are consequently obtained.
\end{proof}
\begin{remark}\label{remark:interior error}Denote 
	$$
	d:=\min \limits_{j}\frac{\left|k^2-E_j\right|}{E_j}\,,
	$$ where $E_j$ are the domain's Dirichlet eigenvalues. 
	From \cite[Eq. (7)]{Kuttler1978}, the interior error of the solution is controlled by
	$$
	\left\|u-u^{(\mu)}\right\|_{L^2(D)} \leqslant \frac{C_{\Omega}}{d}\left\|u-u^{(\mu)}\right\|_{L^2(\partial D)}\leqslant CR^{-M}\|u\|_{L^{2}(\Omega_{\rho})},
	$$
	where $C_{\Omega}$ is a domain-dependent constant and $C$ depends on $k$, $d$, $R$ and $\Omega$, but independent of $M$.
\end{remark}

Finally, we introduce the Runge approximation. A variant of the classical Runge approximation property for uniformly elliptic equations states that if $\Omega_1, \Omega_2$ are bounded, open Lipschitz sets (e.g. balls) with $\Omega_1 \subset\subset \Omega_2$, it is possible to approximate solutions to an equation in the smaller domain $\Omega_{1}$ by solutions of the same equation in the larger domain $\Omega_{2}$. The quantitative Runge approximation indicates that the approximated solution on the larger domain has the exponential bound.
Furthermore, if the solution itself is assumed to be
a solution in a slightly larger domain, one obtains the polynomial bound instead \cite{Salo2019}.
\begin{lemma}\label{lem:runge_quan}
	(Quantitative Runge's approximation \cite{Salo2019}). Let $\tilde{\Omega}$ be a domain with $\Omega_1 \subset\subset \tilde{\Omega} \subset\subset$ $\Omega_2$ where $\Omega_1, \Omega_2$ are bounded, open Lipschitz sets. Denote the elliptic operator $L=\partial_i a^{i j} \partial_j+c$ with symmetric coefficient matrix $(a_{i j})\in L^{\infty}\left(\Omega_2, \mathbb{R}^{n \times n}\right)$ and $c\in L^{\infty}(\Omega_{2})$, and for some $K \geqslant 1$ such that 
	$$
	\begin{aligned}
		&K^{-1}|\xi|^2 \leqslant a^{i j}(x) \xi_i \xi_j \leqslant K|\xi|^2\text{ for a.e. $x \in \Omega_2$ and for all $\xi \in \mathbb{R}^n$},\\
		&\|c\|_{L^{\infty}\left(\Omega_2\right)} \leqslant K,\\
		&\left\|\nabla a^{i j}\right\|_{L^{\infty}\left(\Omega_2\right)} \leqslant K,\quad n \geqslant 3.
	\end{aligned}
	$$
	Define
	$$
	\begin{aligned}
		&S_1:=\left\{w \in H^1\left(\Omega_1\right): L w=0 \text { in } \Omega_1\right\}\\
		&S_{2}:=\left\{h \in H^1\left(\Omega_2\right): L h=0 \text { in } \Omega_2,\left.h\right|_{\partial \Omega_2} \in \tilde{H}^{1 / 2}(\Gamma)\right\},
	\end{aligned}
	$$
	where
	$$
	\tilde{H}^{1 / 2}(\Gamma)=\text { closure of }\left\{g \in H^{1 / 2}\left(\partial \Omega_2\right) ; \operatorname{supp}(g) \subset \Gamma\right\} \text { in } H^{1 / 2}\left(\partial \Omega_2\right)
	$$
	with $\Gamma$ being an open subset of $\partial \Omega_2$. It states that the boundary value of the solution on $\Omega_2$ vanishes outside $\Gamma$. Assume that 0 is not a Dirichlet eigenvalue for $L$ in $\Omega_{2}$.
	Then, for any $\varepsilon \in(0,1)$, there exists a constant $C$ (depending on $\Omega_1, \Omega_2, \tilde{\Omega}$, $c$, $n$, $K$) such that

	(i) for $h \in S_{1}$, there exists $\theta>0$ and $u \in S_{2}$ with
	\begin{equation}\label{eq:runge_case3}
		\left\|h-\left.u\right|_{\Omega_1}\right\|_{L^2\left(\Omega_1\right)} \leqslant \varepsilon\|h\|_{H^1\left(\Omega_1\right)}, \quad\|u\|_{H^{1 / 2}\left(\partial \Omega_2\right)} \leqslant C e^{C \varepsilon^{-\theta}}\|u\|_{L^2\left(\Omega_1\right)}.
	\end{equation}
	
	(ii) For $\tilde{h} \in H^1(\tilde{\Omega})$ with $L \tilde{h}=0$ in $\tilde{\Omega}$,  there exists $\beta\geqslant 1$ and $u \in S_{2}$ with
	\begin{equation}\label{eq:runge_case2}
		\left\|\tilde{\left.h\right|}_{\Omega_1}-\left.u\right|_{\Omega_1}\right\|_{L^2\left(\Omega_1\right)} \leqslant \varepsilon\|\tilde{h}\|_{H^1(\tilde{\Omega})}, \quad\|u\|_{H^{1 / 2}\left(\partial \Omega_2\right)} \leqslant C \varepsilon^{-\beta}\left\|\tilde{\left.h\right|}_{\Omega_1}\right\|_{L^2\left(\Omega_1\right)}.
	\end{equation}
\end{lemma}

\section{Theoretical results}\label{sec:theoretical results}
In this section, the error estimates for learning based algorithm introduced in Section 2 are investigated, including the three different cases, respectively. To this end, the following assumption is necessary.
\begin{assumption}\label{ass}
	Suppose that $k^{2}$ is not any Dirichlet eigenvalue of the opposite of the Laplacian $-\Delta$ for the domain as in Lemma \ref{lem:existence}. Let $\Omega$ be a bounded piecewise smooth domain such that $\Omega\subset\subset D$, where $D=\{x:|x|\leqslant 1\}$ is the unit disk in $\mathbb{R}^n$.
	Let $\left\{x_j\right\}_{j=1}^N$ the collocation points equally placed on $\partial \Omega$. Besides,  $\left\{\hat{x}_i\right\}_{i=1}^M$ be poles equally distributed on a larger circle $\partial O_{R}$ with radius $R>1$, where $D \subset\subset O_{R}$, see Figure \ref{fig:domains} for illustration.
\end{assumption}

\begin{lemma}\label{analytic-domain}
	Assume that $u(\bm{x})$ satisfies Helmholtz equation on $\tilde{\Omega}$, and $O_1\subset\subset O_\rho \subset\subset \tilde{\Omega}$ with $\rho>1$, then there exists $f(z)$ that is analytic in $\{z\in\mathbb{C}|  1/\rho <|z|<\rho\}$ satisfying
	\[
	f(e^{\mathrm{i}\theta})=u(\bm{x})|_{\partial O_1}.
	\]
	Moreover, for $f_1(\theta)=f(e^{\mathrm{i}\theta})$, there exists a constant $C$ depending only on $k$ and $\rho$, such that
	\[
	\hat{f}_1(m)\leq \frac{C}{\rho^{m}}\|u\|_{L^2(\tilde{\Omega})}.
	\]
\end{lemma}
For its proof, see Appendix \ref{app:2}.

\begin{figure}[h]
	\centering  
	\includegraphics[width=0.5\textwidth]{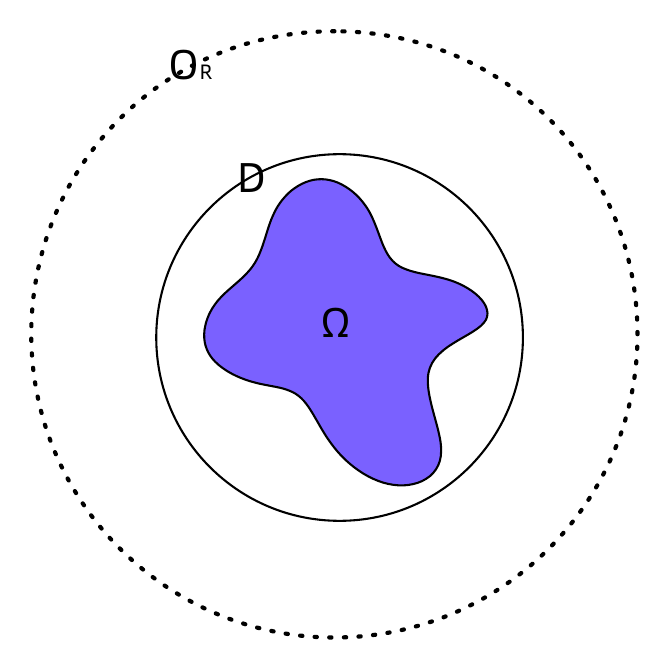}
	\caption{Illustration of domains.}
	\label{fig:domains}
\end{figure}

\subsection{Case 1: The solution can be analytically continued to a sufficiently larger domain.}
In this case, we assume that $u$ satisfies Helmholtz equation \eqref{eq:Helm} in a sufficiently large domain $\Omega_{\rho}$, such that
$$
\Omega \subset\subset D \subset\subset O_{R} \subset\subset \Omega_{\rho}\,,
$$ 
see Fig. \ref{fig:case1} for illustration.
The following theorem indicates a Lipschitz-type convergence estimate about the error.
\begin{figure}[h]
	\centering  
	\includegraphics[width=0.5\textwidth]{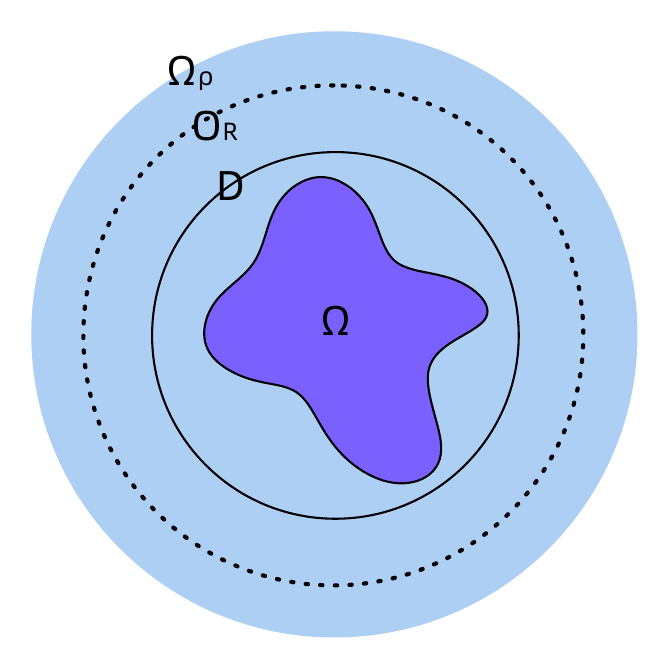}
	\caption{Illustration of Case 1.}
	\label{fig:case1}
\end{figure}

\begin{theorem}\label{thm:error_case1}
	(Error estimate of case 1). Under Assumption \ref{ass},
	assume that there exists a sufficiently large domain $\Omega_{\rho}$, such that
	$$
	\Omega \subset\subset D \subset\subset O_{R} \subset\subset \Omega_{\rho}
	$$
	and $u$ satisfies \eqref{eq:Helm} in $\Omega_{\rho}$, in addition, assume 
	$$
	\|u\|_{H^1(\Omega_{\rho})} \leqslant C_{\rho}
	$$
	with a constant $C_\rho$. When the fundamental solutions $\left\{\Phi\left(\hat{x}_i, x\right)\right\}_{i=1}^M$ are used as learning solutions, denote 
	$$\boldsymbol{b}_x=\left(\Phi\left(\hat{x}_1, x\right), \cdots, \Phi\left(\hat{x}_M, x\right)\right) ,\quad V=\left(\begin{array}{ccc}
		\Phi\left(\hat{x}_1, x_1\right) & \cdots & \Phi\left(\hat{x}_M, x_1\right) \\
		\vdots & \ddots & \vdots \\
		\Phi\left(\hat{x}_1, x_N\right) & \cdots & \Phi\left(\hat{x}_M, x_N\right)
	\end{array}\right)\, ,$$
	$x\in \overline{\Omega}$, and
	\begin{align}\label{eq:Tikhonov2}
		\bm{a}^{(x)}_{x,*}=\underset{\bm{a} \in \mathbb{R}^N}{\arg \min } \left\{\|\bm{a}  V-\bm{b}_x\|^{2}+\alpha \|\bm{a}\|^2\right\}
	\end{align}
	with the regularization parameter
	\begin{equation}\label{eq:choice_alpha}
		\alpha \sim  MNR^{-2M}.
	\end{equation}
	Then, the numerical solution $u^{(c)}(x)=\bm{a}_{*,x}\bm{f}$ has the following error estimate
	$$
	\|u(x)-u^{(c)}(x)\|_{L^{2}(\partial\Omega)}
	\leqslant C\|u\|_{L^{2}(\Omega_{\rho})}\left(R^{-M}+\frac{1}{N}\right)\,,
	$$
	where $C$ depends on the wave number $k$, $d$ and the domain $\Omega$.
\end{theorem}
\begin{proof}
	According to \eqref{eq:Tikhonov} and \eqref{eq:minimum}, it follows that
	$$
	u^{(c)}(x)=\bm{a}_{*,x}\bm{f}=\bm{b}_x V^{*}(VV^{*}+\alpha I)^{-1}\bm{f}
	=\bm{b}_x(V^{*}V+\alpha I)^{-1}V^{*}\bm{f},
	$$
	Denote $\bm{c}_{*}=(V^{*}V+\alpha I)^{-1}V^{*}\bm{f}$, then 
	\begin{align}\label{eq:c}
		\bm{c}_{*}=\underset{\bm{c} \in \mathbb{R}^M}{\arg \min } \left\{\|\bm{f}-V\bm{c}\|^2+\alpha \|\bm{c}\|^2\right\},
	\end{align}
	and $u^{(c)}$ can be equivalently written as 
	$u^{(c)}(x)=\bm{b}_x\bm{c}_{*}$.
	
	By Lemma \ref{analytic-domain}, the conditions in Lemma \ref{FMS} are satisfied. Referring to Lemma \ref{FMS}, Proposition \ref{bound-mu} and Remark \ref{remark:interior error}, there exists fundamental solution approximation $u^{(\mu)}(x)=\bm{b}_x\bm{\mu}$ such that 	
	$$
	\|\bm{\mu}\|^2\leqslant \frac{C}{M}\|u\|_{L^{2}(\Omega_{\rho})}^{2},
	$$
	and
	$$
	\|u-u^{(\mu)}\|_{L^{2}(D)}\leqslant
	CR^{-M}\|u\|_{L^2{(\Omega_{\rho})}}\,.
	$$
	Combining Lemma \ref{lem_interior estimate} and the Sobolev embedding theorem yields
	$$
	|u(x_{j})-u^{(\mu)}(x_{j})|\leqslant 
	CR^{-M}\|u\|_{L^2{(\Omega_{\rho})}},\quad j=1,2,\cdots,N\,.
	$$
	Then, taking $\bm{\mu}$ as candidate, referring to the minimality of \eqref{eq:c}, we have
	$$
	\begin{aligned}
		\left\|\bm{f}-V\bm{c}_{*}\right\|^2 \leqslant \left\|\bm{f}-V\bm{\mu}\right\|^2+\alpha\left\|\bm{\mu}\right\|^2 
		&\leqslant \sum\limits_{j=1}^{N}\left|u(x_{j})-\sum\limits_{i=1}^{M}\Phi(\hat{x}_{i},x_{j})\mu_{i}\right|^2+\alpha \sum\limits_{i=1}^{M}|\mu_{i}|^2\\
		&\leqslant 
		CNR^{-2M}\|u\|_{L^{2}(\Omega_{\rho})}^2
	\end{aligned}
	$$
	and 
	$$
	\|\bm{c}_{*}\|^2\leqslant \frac{1}{\alpha}\|\bm{f}-V\bm{\mu}\|^2+\|\bm{\mu}\|^2\leqslant \frac{C}{M}\|u\|_{L^{2}(\Omega_{\rho})}^{2}.
	$$
	Suppose $\partial \Omega=\cup_{i=1}^{N}\Gamma_{i}$. Then,
	\begin{equation}\label{eq:deduction}
		\begin{aligned}
			\|u(x)-u^{(c)}(x)\|_{L^{2}(\partial\Omega)}^{2}
			\leqslant &\sum\limits_{j=1}^{N}|\Gamma_{j}|\left|u(x_{j})-\sum\limits_{i=1}^{M}\Phi(\hat{x}_{i},x_{j})c_{i}\right|^2+Res
			\\
			\leqslant &C\left(R^{-2M}+\frac{1}{N^2}\right)\|u\|_{L^{2}(\Omega_{\rho})}^2\,,
		\end{aligned}
	\end{equation}
	in which
	$$Res=4|\partial \Omega|\left(\max |\Gamma_{j}|\right)^2
	\left(\underset{x\in\partial \Omega}{\max}|u^{\prime}(x)|^2+M\underset{x\in\partial \Omega,\ \hat{x}\in \partial O_{R}}{\max}|\Phi_{x}(\hat{x},x)|^2\|\bm{c}_{*}\|^2\right).
	$$
	The detailed deduction of \eqref{eq:deduction} can be found in Appendix \ref{app:1}. The last inequality is based on the assumption that the collocation points are equally placed on $\partial\Omega$. Moreover, the Lemma \ref{lem_interior estimate} and the Sobolev embedding theorem give
	$$
	\underset{x\in\partial \Omega}{\max}|u^{\prime}(x)|\leqslant C \|u\|_{L^{2}(\Omega_{\rho})}\,.
	$$
	The proof is complete.
\end{proof}
\subsection{Case 2: The solution can be analytically continued to a slightly larger domain.}\label{sec:case2}
In this case, we assume that $u$ satisfies Helmholtz equation \eqref{eq:Helm} in $\Omega_{\rho}$ with 
$$
\Omega \subset\subset \Omega_{\rho} \subset\subset O_{R}\subset\subset \Omega_{\varepsilon}\,,
$$ 
which means the solution is uniquely continuable to a larger domain than $\Omega$ but may not extended to $O_{R}$. However, Runge approximation tells us that we can find $u_{\varepsilon}$ satisfying the equation in a larger domain $\Omega_{\varepsilon}$ (see Fig. \ref{fig:case2} for example) which approximates $u$ on $\Omega$ with the desired error $\varepsilon$. Lemma \ref{lem:runge_quan} gives the estimate of polynomial type growth of $u_{\varepsilon}$ versus $\varepsilon$. 
The following theorem indicates a H\"{o}lder-type convergence estimate of the present method in this case.
\begin{figure}[h]
	\centering  
	\includegraphics[width=0.48\textwidth]{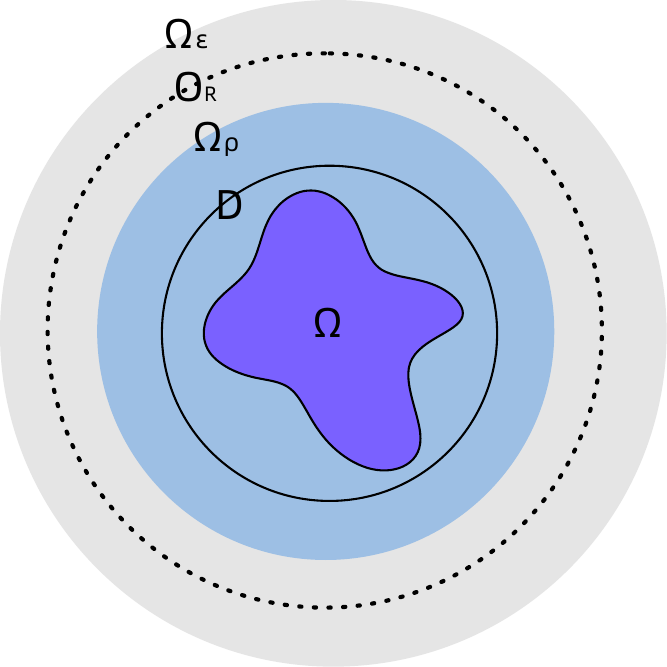}
	\caption{Illustration of Case 2.}
	\label{fig:case2}
\end{figure}
\begin{theorem}\label{thm:error_case2}
	(Error estimate of case 2). Under Assumption \ref{ass},
	we assume that there is $\Omega_{\rho}$ and $\Omega_{\varepsilon}$, such that 
	$$
	\Omega \subset\subset \Omega_{\rho} \subset\subset O_{R}\subset\subset \Omega_{\varepsilon}.
	$$ 
	in addition, assume $u\in H^1(\Omega_\rho)$ satisfying the Helmholtz equation in $\Omega_\rho$ and $u|_{\partial\Omega}=f$. 
	When the fundamental solutions $\left\{\Phi\left(\hat{x}_i, x\right)\right\}_{i=1}^M$ are used as data solution, then the numerical solution
	$u^{(c)}=\bm{a}_{*}\bm{f}$ has the following error estimate
	$$
	\|u(x)-u^{(c)}(x)\|_{L^{2}(\Omega)}
	\leqslant C\|u\|_{H^{1}(\Omega_{\rho})}\left(R^{-M}+\frac{1}{N}\right)^{\kappa},
	$$
	where $\kappa\in(0,1]$ and $C$ depends on $k$, $d$ and the domain $\Omega$, $\Omega_{\varepsilon}$, but independent of $M$.
\end{theorem}
\begin{proof}
	According to Lemma \ref{lem:runge_quan}, there exists $u_{\varepsilon}$ satisfying \eqref{eq:Helm} in $\Omega_{\varepsilon}$ and 
	\begin{equation}\label{eq:case2_proof_runge}
		\|u(x)-u_{\varepsilon}(x)\|_{L^2(\Omega)}
		\leqslant  \varepsilon\|u\|_{H^{1}(\Omega_{\rho})},\quad 
		\|u_{\varepsilon}\|_{H^{1/2}(\partial\Omega_{\varepsilon})}\leqslant C\varepsilon^{-\beta}\|u\|_{L^{2}(\Omega)},
	\end{equation}
	where $\varepsilon\in(0,1)$ and $\beta\geqslant 1$ and $C$ depends on $k$, $\Omega$ and $\Omega_{\varepsilon}$. 
	Furthermore,
	$$
	\|u_{\varepsilon}\|_{H^{1}(\Omega_{\varepsilon})}\leqslant C\varepsilon^{-\beta}\|u\|_{L^{2}(\Omega)},
	$$
	Moreover, the Theorem \ref{thm:error_case1} can be applied to $u_{\varepsilon}$ to obtain the corresponding numerical solution $u_{\varepsilon}^{(c)}$. It follows that
	$$
	\|u_{\varepsilon}(x)-u_{\varepsilon}^{(c)}(x)\|_{L^2(\Omega)}\leqslant C_{1}\left(R^{-M}+\frac{1}{N}\right)\|u_{\varepsilon}\|_{L^{2}(\Omega_{\varepsilon})},
	$$
	where $C_{1}$ depends on $k$, $d$ and $\Omega$.
	Therefore,
	$$
	\begin{aligned}
		&\|u(x)-u^{(c)}(x)\|_{L^2(\Omega)}\\
		\leqslant &
		\|u(x)-u_{\varepsilon}(x)\|_{L^2(\Omega)}+\|u_{\varepsilon}(x)-u_{\varepsilon}^{(c)}(x)\|_{L^2(\Omega)}
		+ \|u_{\varepsilon}^{(c)}(x)-u^{(c)}(x)\|_{L^2(\Omega)} \\
		\leqslant &
		\varepsilon \|u\|_{H^{1}(\Omega_{\rho})}+C_{1}\left(R^{-M}+\frac{1}{N}\right)\varepsilon^{-\beta}\|u\|_{H^{1}(\Omega_{\rho})},
	\end{aligned}
	$$
	Denote $\delta = R^{-M}+\frac{1}{N}$. Let $\varepsilon = \delta\varepsilon^{-\beta}$, we have $\varepsilon = \delta^{\frac{1}{1+\beta}}$.
	Therefore,
	$$
	\begin{aligned}
		\|u(x)-u^{(c)}(x)\|_{L^2(\Omega)}
		\leqslant C\delta^{\frac{1}{1+\beta}}\|u\|_{H^{1}(\Omega_{\rho})}.
	\end{aligned}
	$$
	The proof is complete.
\end{proof}
\subsection{Case 3: The solution cannot be analytically continued to a larger domain}\label{sec:case3}
In this case, we assume that $u$ only satisfies Helmholtz equation \eqref{eq:Helm} in $\Omega$. It states that the solution cannot be uniquely continuable to a larger domain than $\Omega$. Fortunately, Runge approximation tells us that we can find $u_{\varepsilon}$ satisfying the equation in a larger domain $\Omega_{\varepsilon}$ (see Fig. \ref{fig:case3} for example), i.e.,
$$
\Omega\subset\subset D\subset\subset O_{R}\subset\subset \Omega_{\varepsilon},
$$
which approximates $u$ on $\Omega$ with the desired error $\varepsilon$. Lemma \ref{lem:runge_quan} gives the estimate of the exponential type growth for the bound of $u_{\varepsilon}$ on $\Omega_\varepsilon$. The following theorem gives a logarithm-type convergence estimate.
\begin{figure}[h]
	\centering  
	\includegraphics[width=0.48\textwidth]{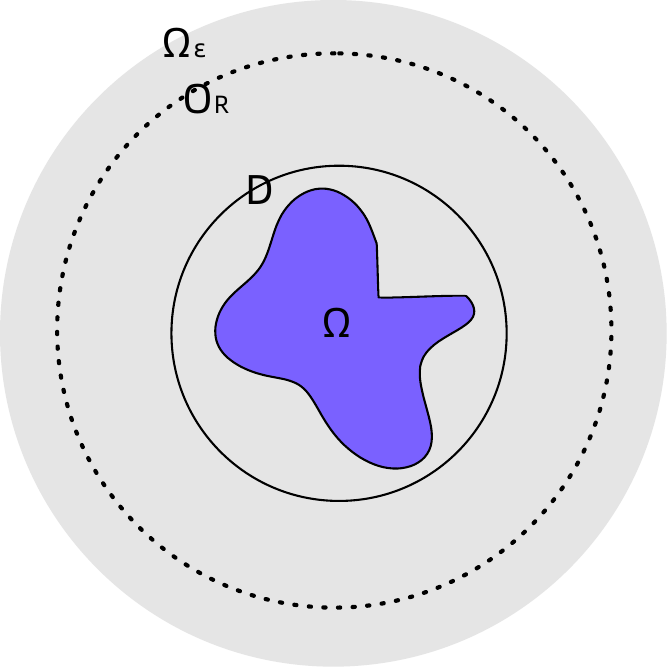}
	\caption{Illustration of Case 3.}
	\label{fig:case3}
\end{figure}
\begin{theorem}\label{thm:error_case3}
	(Error estimate of case 3). Under Assumption \ref{ass},
	we assume that there exists $\Omega_{\varepsilon}$such that 
	$$
	\Omega\subset\subset D\subset\subset O_{R}\subset\subset \Omega_{\varepsilon}.
	$$
	and 
	$$
	\|u\|_{H^1(\Omega)} \leqslant C.
	$$
	When the fundamental solutions $\left\{\Phi\left(\hat{x}_i, x\right)\right\}_{i=1}^M$ are used as data solutions, then the numerical solution
	$u^{(c)}=\bm{a}_{*}\bm{f}$ has the following error estimate
	$$
	\|u(x)-u^{(c)}(x)\|_{L^{2}(\Omega)}
	\leqslant C\|u\|_{H^{1}(\Omega)}\frac{1}{\left|\ln{\left(R^{-M}+\frac{1}{N}\right)}\right|^{\eta}},
	$$
	where $\eta>0$ and $C$ depends on $k$, $d$ and the domain $\Omega$, $\Omega_{\varepsilon}$.
\end{theorem}
\begin{proof}
	According to Lemma \ref{lem:runge_quan}, there exists $u_{\varepsilon}$ satisfying \eqref{eq:Helm} in $\Omega_{\varepsilon}$ such that \begin{equation}\label{eq:case3_proof_runge}
		\|u(x)-u_{\varepsilon}(x)\|_{L^2(\Omega)}
		\leqslant  \varepsilon\|u\|_{H^{1}(\Omega)},\quad 
		\|u_{\varepsilon}\|_{H^{1/2}(\partial\Omega_{\varepsilon})}\leqslant C\mathrm{e}^{C\varepsilon^{-\theta}}\|u\|_{L^{2}(\Omega)},
	\end{equation}
	where $\varepsilon \in(0,1)$, $\theta>0$ and $C$ depends on $k$, $\Omega$ and $\Omega_{\varepsilon}$. 
	Furthermore,
	$$
	\|u_{\varepsilon}\|_{H^{1}(\Omega_{\varepsilon})}\leqslant C\mathrm{e}^{C\varepsilon^{-\theta}}\|u\|_{L^{2}(\Omega)}
	$$
	Besides, we can apply Theorem \ref{thm:error_case1} to $u_{\varepsilon}$ to obtain the corresponding numerical solution $u_{\varepsilon}^{(c)}$. It follows that
	$$
	\|u_{\varepsilon}(x)-u_{\varepsilon}^{(c)}(x)\|_{L^2(\Omega)}\leqslant C_{2}\left(R^{-M}+\frac{1}{N}\right)\|u\|_{L^{2}(\Omega_{\varepsilon})},
	$$
	where $C_{2}$ depends on $k$, $d$ and the domain $\Omega$.
	Therefore,
	$$
	\begin{aligned}
		&\|u(x)-u^{(c)}(x)\|_{L^2(\Omega)}\\
		\leqslant &
		\|u(x)-u_{\varepsilon} (x)\|_{L^2(\Omega)}+\|u_{\varepsilon}(x)-    u_{\varepsilon}^{(c)}(x)\|_{L^2(\Omega)}
		+ \|u_{\varepsilon}^{(c)}(x)-u^{(c)}(x)\|_{L^2(\Omega)} \\
		\leqslant &
		\varepsilon \|u\|_{H^{1}(\Omega)}+C_{2}\left(R^{-M}+\frac{1}{N}\right)e^{C\varepsilon^{-\theta}}\|u\|_{H^{1}(\Omega)}.
	\end{aligned}
	$$
	Denote $\delta = R^{-M}+\frac{1}{N}$
	and take 
	$$
	\varepsilon=\left(\frac{C}{\gamma|\ln \delta|}\right)^{\frac{1}{\theta}},\quad \gamma \in(0,1).
	$$ 
	Thus,
	$$
	C \varepsilon^{-\theta}=\gamma|\ln \delta|.
	$$
	Therefore, we have 
	$$
	\begin{aligned}
		\|u(x)-u^{(c)}(x)\|_{L^2(\Omega)}
		\leqslant C\frac{1}{|\ln \delta|^{\frac{1}{\theta}}}.
	\end{aligned}
	$$			
	The proof is complete.
\end{proof}

\section{Numerical examples}\label{sec:numerical examples}
In this section, we implement amounts of numerical experiments for Helmholtz equation with high wave number in complex regions to verify our algorithm and theoretical convergence analysis.

We consider the boundary value problem for Helmholtz equation with the high wave number $k=\frac{2\pi freq}{340}$ when the frequency $freq=10000Hz$ with $k\approx 184.79$. Let $\Omega$ be the flower-shape region whose boundary $\partial \Omega$ can be parameterized by 
$$
(a \cos t - b \cos nt \cos t, a \sin t - b \cos nt\sin t)
$$ with $t\in [0, 2\pi]$ and $(a,b,n)=(0.5,b=0.1,6)$.

\subsection{Case 1}
We choose the exact solution 
$$
u(x)=\sin(\frac{k}{\sqrt{2}}x)\sin(\frac{k}{\sqrt{2}}y)
$$ 
as shown in \subref{fig:case1_exact} of Fig. \ref{fig:case1_exact_source}. It is easy to find that the exact solution can be continued to a sufficiently large domain and oscillates severely as the wave number $k$ grows.

We set $R=1.07$ and choose 288 points on the boundary as the collocation points, i.e., $N=288$. Besides, $M$ is set to be identically 288. Based on the choice of regularization parameter in \eqref{eq:choice_alpha}, we can choose $\alpha = 1\times 10^{-12}$. In Fig. \ref{fig:case1_exact_source} \subref{fig:case1_source}, a number of blue dots represent the locations of 37500 solving points, the red star points represent 408 collocation points on the boundary and the circles surrounding the flower-shape denote the pole points.

\begin{figure}[h]
	\centering  
	\subfigure[The exact solution for Case 1.
	]{
		\label{fig:case1_exact}
		\includegraphics[width=0.48\textwidth]{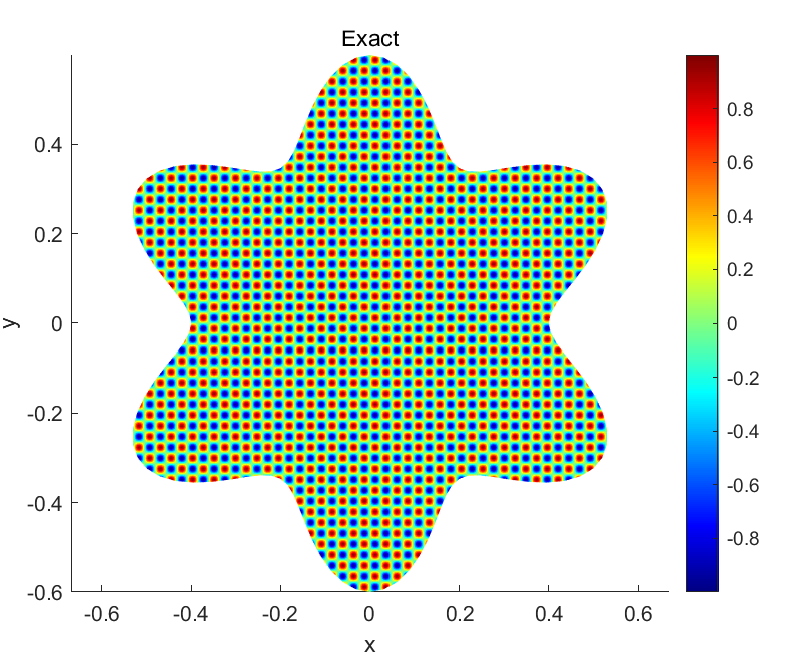}}
	\subfigure[Points settings.]{
		\label{fig:case1_source}
		\includegraphics[width=0.48\textwidth]{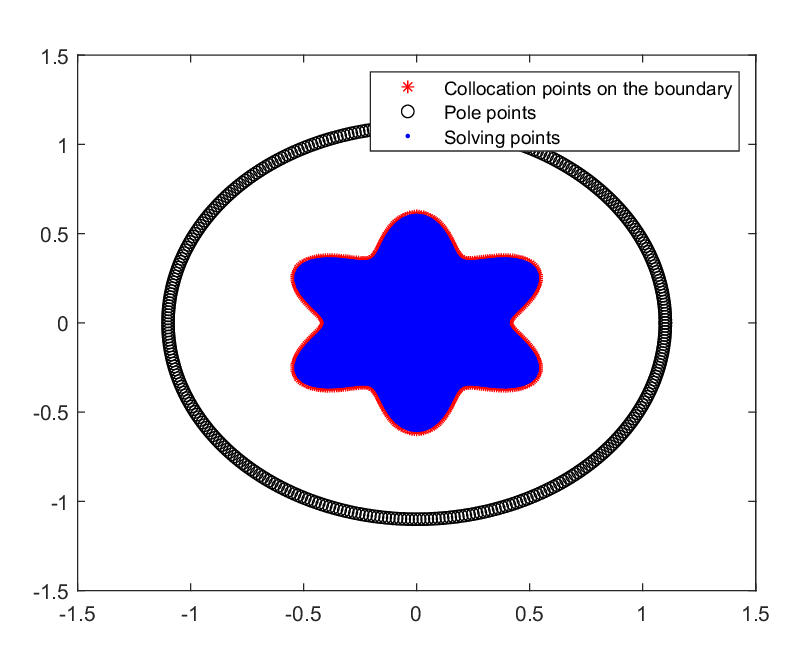}}
	\caption{Exact solution and points settings for Case 1.}
	\label{fig:case1_exact_source}
\end{figure}
Fig \ref{fig:case1_ngm_solution_error} shows the numerical solution and the corresponding error for our algorithm.
\begin{figure}[h]
	\centering  
	\subfigure[Numerical solution for Case 1]{
		\label{fig:case1_ngm_solution}
		\includegraphics[width=0.48\textwidth]{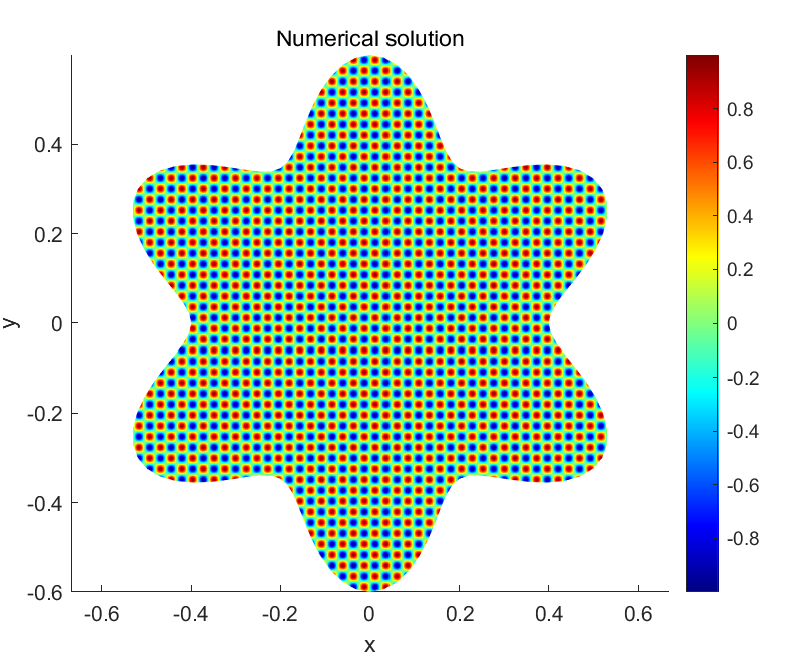}}
	\subfigure[Error]{
		\label{fig:case1_ngm_error}
		\includegraphics[width=0.48\textwidth]{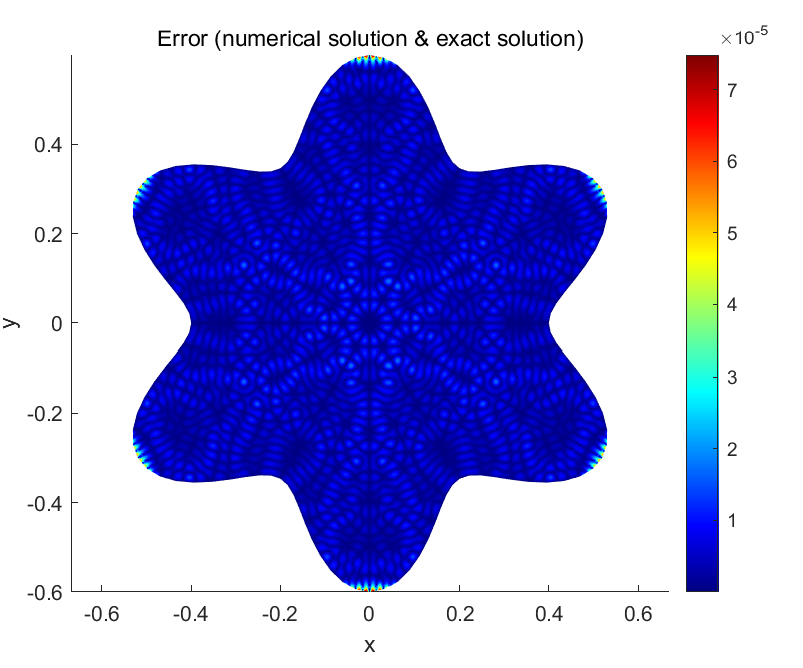}}
	\caption{Numerical solution and error of our algorithm for Case 1}
	\label{fig:case1_ngm_solution_error}
\end{figure}
Table \ref{tab:case1} gives the numerical results for Case 1. It costs us 2.2823 seconds to learn the discrete solution operator. 

\begin{remark}
	We also use Partial Differential Equation Toolbox in Matlab to solve Case 1. The target maximum element edge length and the target minimum element edge length are 0.001 and 0.0005, respectively. It spends 178.0038 seconds and $2$-norm and $\infty$-norm of the error vector are 41.0275 and 0.8555, respectively. 
\end{remark}

\begin{table}[h]
	\caption{Numerical results of Case 1}
	\label{tab:case1}
	\centering
	\tabcolsep=8pt
	\begin{tabular}{cc}
		\hline
		\hline
		$2$-norm for error vector&  1.0582e-03 \\
		$\infty$-norm for error vector & 7.4813e-05 \\
		Time (seconds)&0.0069\\
		\hline
		\hline
	\end{tabular}
\end{table}

\subsection{Case 2}
In this case, using the operator learned in Case 1, we verify the solution that can be continued to a slightly larger domain. For example, we choose the exact solution 
$$
u(x)=\Phi(\hat{x},x),
$$
where $\hat{x} = (0.55,0)$, as shown in Fig. \ref{fig:case2_exact}. 
\begin{figure}[h]
	\centering  
	\includegraphics[width=0.48\textwidth]{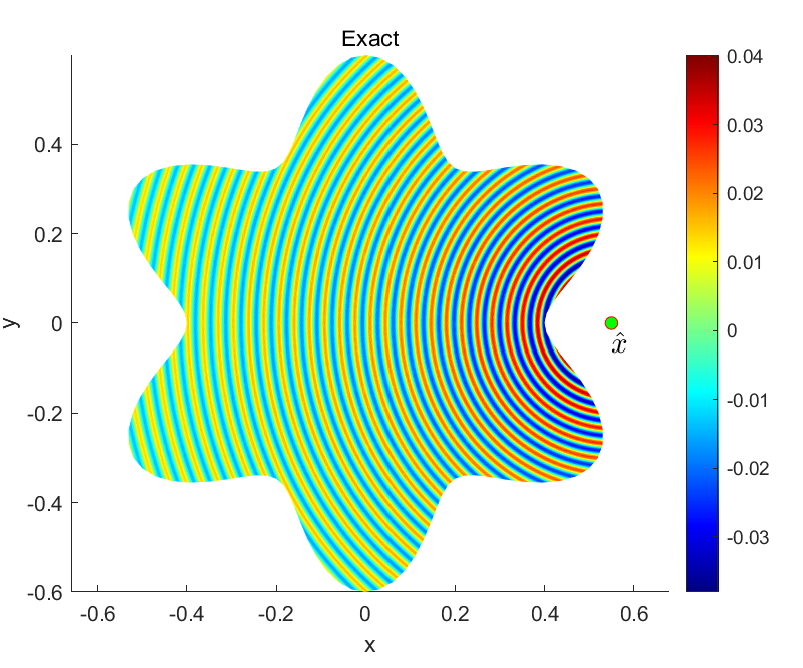}	\caption{Exact solution for Case 2.}   \label{fig:case2_exact}
\end{figure}

The numerical solution and the error of our numerical method are shown in Fig.
\ref{fig:case2_num_results}. 
\begin{figure}[h]
	\centering  
	\subfigure[Numerical solution for Case 2]{\label{fig:case2_num}
		\includegraphics[width=0.48\textwidth]{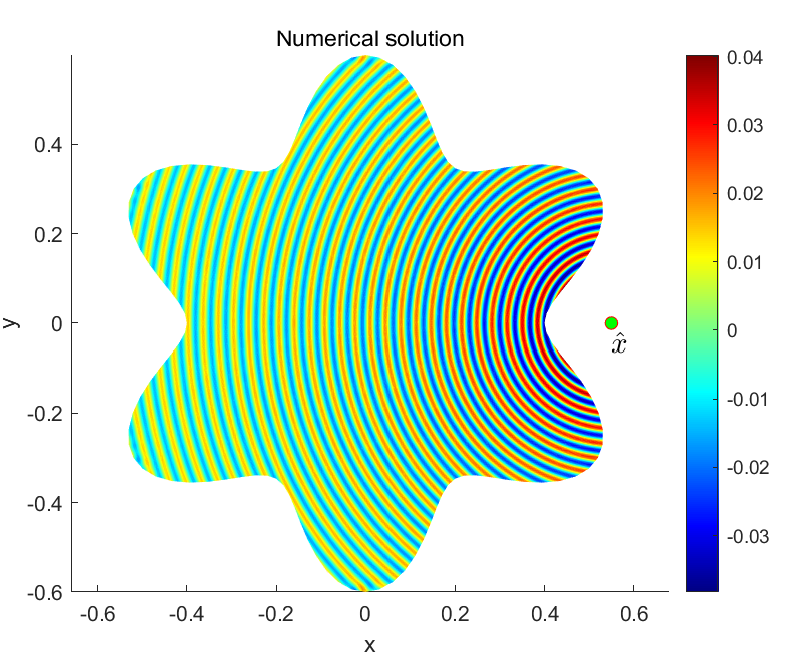}}
	\subfigure[Error for Case 2]{
		\label{fig:case2_error}
		\includegraphics[width=0.48\textwidth]{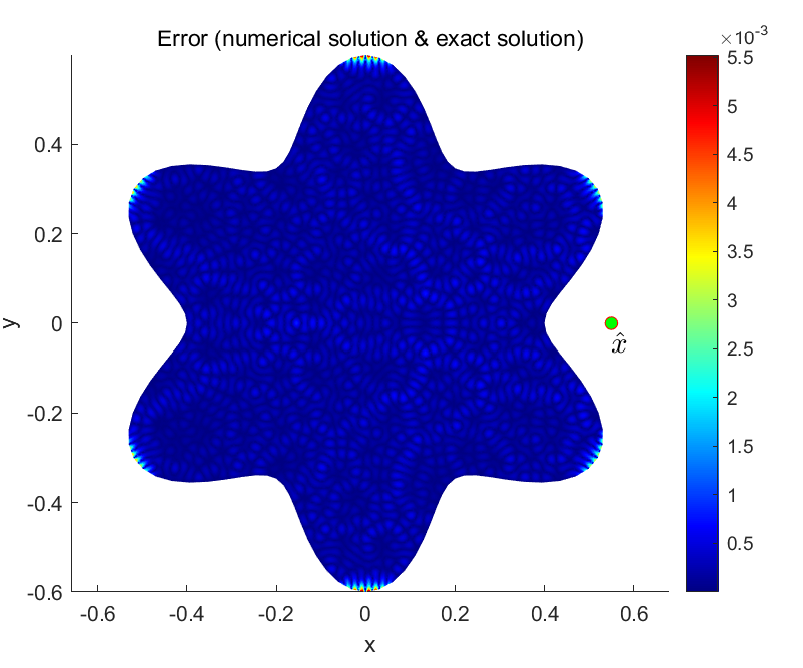}}
	\caption{Numerical solution and error of our algorithm for Case 2}    \label{fig:case2_num_results}
\end{figure}
Table \ref{tab:case2} gives the numerical results for Case 2. Compared with the results of Case 1 in Table \ref{tab:case1}, the error grows which coincides with Theorem \ref{thm:error_case2}.
\begin{table}[H]
	\centering
	\caption{Numerical results for Case 2.}
	\label{tab:case2}
	\centering
	\tabcolsep=8pt
	\begin{tabular}{cc}
		\hline
		\hline
		$2$-norm for error vector&  5.5929e-02 \\
		$\infty$-norm for error vector & 5.5240e-03 \\
		Time (seconds)&0.0064\\
		\hline
		\hline
	\end{tabular}
\end{table}
\begin{remark}
	We also use Partial Differential Equation Toolbox in Matlab to solve Case 2. The target maximum element edge length and the target minimum element edge length are 0.001 and 0.0005, respectively. It spends 182.0038 seconds and $2$-norm and $\infty$-norm of the error vector are 2.7629 and 0.0517, respectively.
	The reason that the error of this case is less than that of Case 1 is that the value of the exact solution is relatively small.
\end{remark}
\subsection{Case 3}
In this case, we verify the solution that the singularity occurs closer to the region $\Omega$ compared with Case 2. For example, we choose the exact solution 
$$
u(x)=\Phi(\hat{x}_{1},x)-\Phi(\hat{x}_{2},x),
$$ 
where $\hat{x}_{1} = (0.45,0.05)$, $\hat{x}_{2} = (0.45,-0.05)$, as shown in Fig. \ref{fig:case3_exact_error} \subref{fig:case3_exact}. Besides, Fig. \ref{fig:case3_exact_error} \subref{fig:case3_error} gives the error of our algorithm. The error becomes larger, which coincides with Theorem \ref{thm:error_case3}. How to improve the accuracy for such cases is our additional work.
\begin{figure}[h]
	\centering  
	\subfigure[Exact solution for Case 3]{
		\label{fig:case3_exact}
		\includegraphics[width=0.48\textwidth]{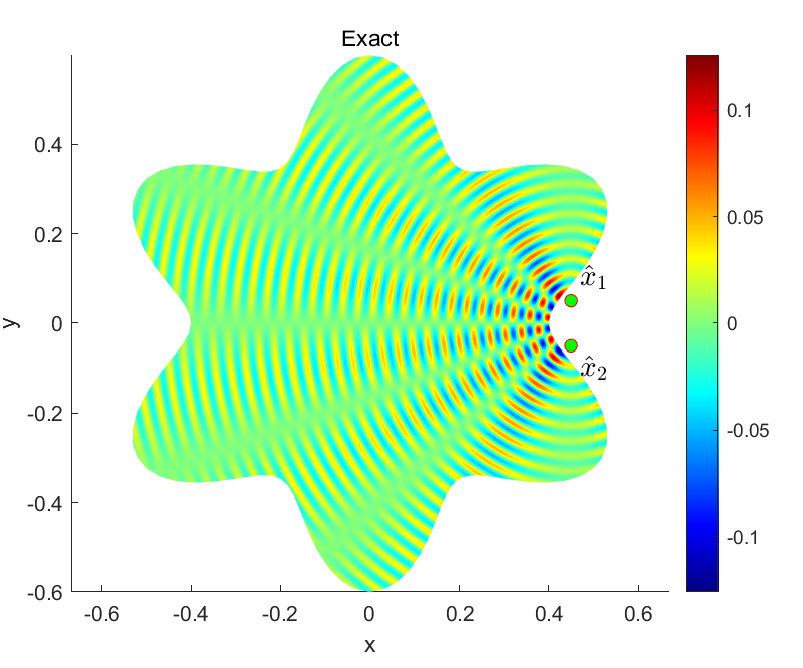}}
	\subfigure[Error for Case 3]{
		\label{fig:case3_error}
		\includegraphics[width=0.48\textwidth]{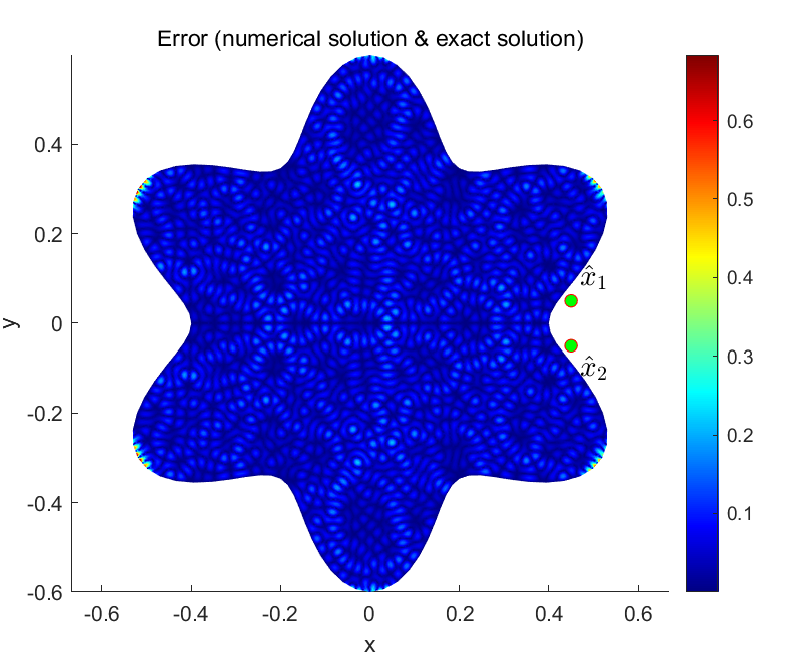}}
	\caption{Exact solution and error of our algorithm for Case 3}
	\label{fig:case3_exact_error}
\end{figure}
\section{Concluding remarks}\label{sec:concluding}
In this work, we propose a learning based numerical method (LbNM) for high-frequency Helmholtz equation.
It is based on various solution data, especially the fundamental solutions,  to reconstruct the solution mapping. Different from the popular machine learning methods, we give the complete error estimates for our learning based numerical methods, which means it has strong interpretability and generalizability.

In terms of theoretical results, the main tools are the MFS and Runge approximation. We have divided into three cases based on the region that the solution can be continued.
We employ the the error estimates for the MFS and Tikhonov regularization to achieve the Lipschitz-type error estimate for the first case. The combination Case 1 with the quantitative version of Runge’s approximation gives the H\"{o}lder-type error estimates for Case 2. Correspondingly, the logarithm-type error estimates could be obtained for Case 3.

Numerical examples have shown the efficient and high-precision performance of our algorithm in high-frequency regimes and verifies the theoretical results for three cases, respectively.

In summary, the present method is mesh-free and independent on the interior discrete error. The computational precision mainly depends on the number of collocation points on the boundary, the number of the source points and the locations that those poles are placed. It means that the reconstruction of the solution operator would be more reliable with increasing data.
The reconstructions of solution operators with respect to different objective location can be carried out individually. It is beneficial to make the computation flexible when the engineers are only interested in local areas in engineering problems.
In practice, we can also adopt other kind of data such as existing numerical
solutions and even experimental measurements. 

In the future, we will focus on the novel numerical methods to improve the precision of Case 3 of section \ref{sec:case3}. Besides, the corresponding theoretical analysis for the learning based numerical methods on mixed boundary problems should be complete.  More generally, the inhomogeneous problems and the problems with variable coefficients should be considered as well. Furthermore, it is possible to apply the learning based numerical methods widely to other problems such as Cauchy problems \cite{Cheng2023}.




\textbf{Acknowledgement}


This work was supported by the National Natural Science Foundation of China (Nos. 12201386, 12241103) and the National Nature Science Foundation of China under Grant 12126601 and the R\&D project of Pazhou Lab (Huangpu) under Grant 2023K0608.

Thanks are due to Dr. Min Zhong from Southeastern University for valuable discussions.


\bibliographystyle{unsrt} 
\bibliography{ref.bib}

\begin{thebibliography}{10}

\bibitem{Isakov2004}
T.~Hrycak and V.~Isakov.
\newblock Increased stability in the continuation of solutions to the
  {H}elmholtz equation.
\newblock {\em Inverse Problems}, 20:697, 03 2004.

\bibitem{Nagayasu2013}
S.~Nagayasu, G.~Uhlmann, and J.~N. Wang.
\newblock Increasing stability in an inverse problem for the acoustic equation.
\newblock {\em Inverse Problems}, 29(2):025012, 2013.

\bibitem{Cheng2016}
J.~Cheng, V.~Isakov, and S.~Lu.
\newblock Increasing stability in the inverse source problem with many
  frequencies.
\newblock {\em Journal of Differential Equations}, 260(5):4786--4804, 2016.

\bibitem{Isakov2020}
M.~Entekhabi and V.~Isakov.
\newblock Increasing stability in acoustic and elastic inverse source problems.
\newblock {\em SIAM Journal on Mathematical Analysis}, 52(5):5232--5256, 2020.

\bibitem{Mathon1977}
R.~Mathon and R.~L. Johnston.
\newblock The approximate solution of elliptic boundary-value problems by
  fundamental solutions.
\newblock {\em SIAM Journal on Numerical Analysis}, 14(4):638--650, 1977.

\bibitem{Alexander2020}
Alexander~H.D. Cheng and Y.~X. Hong.
\newblock An overview of the method of fundamental solutions-solvability,
  uniqueness, convergence, and stability.
\newblock {\em Engineering Analysis with Boundary Elements}, 120:118--152,
  2020.

\bibitem{Antunes2018_1}
P.~Antunes.
\newblock A numerical algorithm to reduce the ill conditioning in meshless
  methods for the {H}elmholtz equation.
\newblock {\em Numerical Algorithms}, 79, 11 2018.

\bibitem{Antunes2018_2}
P.~Antunes.
\newblock Reducing the ill conditioning in the method of fundamental solutions.
\newblock {\em Advances in Computational Mathematics}, 44, 02 2018.

\bibitem{Antunes2022}
P.~Antunes.
\newblock A well-conditioned method of fundamental solutions for {L}aplace
  equation.
\newblock {\em Numerical Algorithms}, 91, 04 2022.

\bibitem{Barnett2008}
A.H. Barnett and T.~Betcke.
\newblock Stability and convergence of the method of fundamental solutions for
  {H}elmholtz problems on analytic domains.
\newblock {\em Journal of Computational Physics}, 227(14):7003--7026, 2008.

\bibitem{Lax1956}
P.~D. Lax.
\newblock A stability theorem for solutions of abstract differential equations,
  and its application to the study of the local behavior of solutions of
  elliptic equations.
\newblock {\em Communications on Pure and Applied Mathematics}, 9(4):747--766,
  1956.

\bibitem{Malgrange1956}
B.~Malgrange.
\newblock Existence et approximation des solutions des \'{e}quations aux
  d\'{e}riv\'{e}es partielles et des \'{e}quations de convolution.
\newblock {\em Annales de l'Institut Fourier}, 6:271--355, 1956.

\bibitem{Salo2019}
A.~R\"{u}land and M.~Salo.
\newblock Quantitative runge approximation and inverse problems.
\newblock {\em International Mathematics Research Notices},
  2019(20):6216--6234, 2018.

\bibitem{Ruland2021}
M.~Garc\'{i}a-Ferrero, A.~R\"{u}land, and W.~Zato\'{n}.
\newblock Runge approximation and stability improvement for a partial data
  {C}alder\'{o}n problem for the acoustic {H}elmholtz equation.
\newblock {\em Inverse Problems \& Imaging}, 16:251--281, 01 2021.

\bibitem{Pohjola2022}
Valter Pohjola.
\newblock On quantitative runge approximation for the time harmonic maxwell
  equations.
\newblock {\em Transactions of the American Mathematical Society},
  375(08):5727--5751, 2022.

\bibitem{Kravchenko2021}
V.~Kravchenko and V.~Vicente-Ben\'{i}tez.
\newblock Runge property and approximation by complete systems of solutions for
  strongly elliptic equations.
\newblock {\em Complex Variables and Elliptic Equations}, 67(3):661--682, 2022.

\bibitem{Cheng2005}
J.~{Cheng}, J.~J. {Liu}, and G.~{Nakamura}.
\newblock The numerical realization of the probe method for the inverse
  scattering problems from the near-field data.
\newblock {\em Inverse Problems}, 21(3):839--855, 2005.

\bibitem{Tikhonov1977}
A.~N. Tikhonov and V.~Arsenin.
\newblock {\em Solutions of ill-posed problems}.
\newblock Wiley, New York, 1977.

\bibitem{Beretta2016}
E.~Beretta, M.~V. de~Hoop, F.~Faucher, and O.~Scherzer.
\newblock Inverse boundary value problem for the {H}elmholtz equation:
  quantitative conditional {L}ipschitz stability estimates.
\newblock {\em SIAM Journal on Mathematical Analysis}, 48(6):3962--3983, 2016.

\bibitem{Evans2016}
L.~C. Evans.
\newblock {\em Partial differential equations, Second edition}.
\newblock American Math. Society, 2016.

\bibitem{Kuttler1978}
J.~R. Kuttler and V.~G. Sigillito.
\newblock Bounding eigenvalues of elliptic operators.
\newblock {\em SIAM Journal on Mathematical Analysis}, 9(4):768--773, 1978.

\bibitem{Cheng2023}
Y.~Chen, J.~Cheng, S.~Lu, and M.~Yamamoto.
\newblock Harmonic measures and numerical computation of {C}auchy problems for
  {L}aplace equations.
\newblock {\em Chinese Annals of Mathematics, Series B}, 44(6):913--928, 2023.

\end{thebibliography}


\clearpage
\pagenumbering{alph}
\begin{appendices}
	\section{Detailed deduction of \texorpdfstring{\eqref{eq:deduction}}{}}
	\label{app:1}
	Denote
	$$
	v_{M}(x)=u(x)-\sum\limits_{i=1}^{M}c_{i}\Phi(\hat{x}_{i},x).
	$$
	We have 
	$$
	\begin{aligned}
		&\int\limits_{\partial\Omega}|v_{M}(x)|^2\mathrm{d}x\\
		\leqslant& 2\int\limits_{\partial\Omega}\big|v_{M}(x)-\sum\limits_{j=1}^{N}v_{M}(x_{j})e_{j}(x)\big|^2\mathrm{d}x+2\int\limits_{\partial\Omega}\big|\sum\limits_{j=1}^{N}v_{M}(x_{j})e_{j}(x)\big|^2\mathrm{d}x\\
		\leqslant&
		2\int\limits_{\partial\Omega}\big|\sum\limits_{j=1}^{N}(v_{M}(x)-v_{M}(x_{j}))e_{j}(x)\big|^2\mathrm{d}x+2\int\limits_{\partial\Omega}\big|\sum\limits_{j=1}^{N}v_{M}(x_{j})e_{j}(x)\big|^2\mathrm{d}x\\
		\leqslant&
		2\int\limits_{\partial\Omega}\sum\limits_{j=1}^{N}|v_{M}(x)-v_{M}(x_{j})|^2e_{j}^2(x)\mathrm{d}x+2\int\limits_{\partial\Omega}\sum\limits_{j=1}^{N}v_{M}^{2}(x_{j})e_{j}^{2}(x)\mathrm{d}x\\
		\leqslant&
		2\sum\limits_{j=1}^{N}\int\limits_{\partial\Omega_{j}}|v_{M}(x)-v_{M}(x_{j})|^2\mathrm{d}x+2\sum\limits_{j=1}^{N}\int\limits_{\partial\Omega_{j}}v_{M}^{2}(x_{j})\mathrm{d}x\\
		\leqslant&
		4|\partial \Omega|\left(\max\limits_{j} |\Gamma_{j}|\right)^2
		\left(\underset{x\in\partial \Omega}{\max}|u^{\prime}(x)|^2+M\underset{x\in\partial \Omega,\ \hat{x}\in \partial O_{R}}{\max}|\Phi_{x}(\hat{x},x)|^2\|\bm{c}\|^2\right)+2\sum\limits_{j=1}^{N}\int\limits_{\Gamma_{j}}v_{M}^{2}(x_{j})\mathrm{d}x.
	\end{aligned}
	$$
	The last inequality results from
	$$
	\begin{aligned}
		&\sum\limits_{j=1}^{N}\int\limits_{\Gamma_{j}}|v_{M}(x)-v_{M}(x_{j})|^2\mathrm{d}x\\
		\leqslant & 2\sum\limits_{j=1}^{N}\int\limits_{\partial\Omega_{j}}|u(x)-u(x_{j})|^2+\big|\sum\limits_{i=1}^{M}c_{i}(\Phi(\hat{x}_{i},x)-\Phi(\hat{x}_{i},x))\big|^2\mathrm{d}x\\
		\leqslant &    2\sum\limits_{j=1}^{N}\int\limits_{\Gamma_{j}}|u(x)-u_(x_{j})|^2+\sum\limits_{i=1}^{M}\big|(\Phi(\hat{x}_{i},x)-\Phi(\hat{x}_{i},x))\big|^2\|\bm{c}\|^2\mathrm{d}x\\
		\leqslant & 2\sum\limits_{j=1}^{N}\int\limits_{\Gamma_{j}}|x-x_{j}|^2|u^{\prime}(\xi_{j})|^2+\sum\limits_{i=1}^{M}|x-x_{j}|^2\big|\Phi_{x}(\hat{x}_{i},\eta_{j})\big|^2\|\bm{c}\|^2\mathrm{d}x\\
		\leqslant &
		2|\partial \Omega|\left(\max\limits_{j} |\Gamma_{j}|\right)^2
		\left(\underset{x\in\partial \Omega}{\max}|u^{\prime}(x)|^{2}+M\underset{x\in\partial \Omega, \ \hat{x}\in \partial O_{R}}{\max}|\Phi_{x}(\hat{x},x)|^{2}\|\bm{c}\|^2\right).
	\end{aligned}
	$$

	\section{Proof of Lemma \ref{analytic-domain}}	\label{app:2}
	
	Since $u(x)$ satisfies Helmholtz equation on $\tilde{\Omega}$, there exists single layer potential $g(\zeta)$ on $\partial O_\rho$ such that,
	\[
	u(\bm{x})=\frac{\mathrm{i}}{4}\int_0^{2\pi}H_0^{(1)}(k|\bm{x}-\rho e^{\mathrm{i}\phi}|)g(\phi)\mathrm{d}\phi.
	\]
	Then 
	\[
	f_1(\theta)=u|_{\partial O_1}=\frac{\mathrm{i}}{4}\int_0^{2\pi}H_0^{(1)}(k| e^{\mathrm{i}\theta}-\rho e^{\mathrm{i}\phi}|)g(\phi)\mathrm{d}\phi .
	\]
	As indicated in \cite{Barnett2008},
	\[
	f_1(\theta)=\int_{0}^{2\pi} \frac{\mathrm{i}}{4}\sum_{m\in\mathbb{Z}} H^{(1)}_m(k\rho)  J_m(kr) e^{\mathrm{i}m(  \theta-\phi   )}g(  \phi )\mathrm{d}\phi ,
	\]
	which is in convolution form and yields
	\[
	\hat{f}_1(m)=\hat{s}(m)\hat{g}(m),
	\]
	with the discrete Fourier coefficients
	\[
	\hat{s}(m)=2\pi \left( \frac{\mathrm{i}}{4}\sum_{n\in\mathbb{Z}} H^{(1)}_n(k\rho)  J_n(kr) e^{\mathrm{i}n\theta}\right)^{\widehat{}}\;(m)=\frac{\pi\mathrm{i}}{2}  H^{(1)}_m(k\rho)  J_m(kr) .
	\]
	Due to asymptotic behavior of $J_m(z), Y_m(z)$ as $|m|\rightarrow \infty$ for fixed $k$ \cite{Barnett2008}, there exists constants $c_s,C_s$ depending only on $\rho,k$ such that
	$$\frac{c_s}{|m|}\rho^{-|m|}\leq |\hat{s}(m)|\leq C_s\rho^{-|m|}, \quad m\in\mathbb{Z}.
	$$
	Therefore,
	\[
	\hat{f}_1(m)=\hat{s}(m)\hat{g}(m)\leq \frac{C_s}{\rho^{m}}\|\hat{g}\|_{l^2(\mathbb{Z})}=\frac{C_s}{2\pi \rho^{m}}\|g\|_{L^2([0,2\pi])}\leq \frac{C}{\rho^{m}}\|u\|_{L^2(\partial\tilde{\Omega})}.
	\]
	Define $f(z)$ by the following Laurent series
	\[
	f(z)=\sum_{n=-\infty}^\infty c_n z^n,
	\]
	with
	\[
	c_n=\frac{1}{2\pi\mathrm{i}}\int_{|\zeta|=1}\frac{u(\zeta)}{(\zeta-0)^{n+1}}\mathrm{d}\zeta  =\frac{1}{2\pi}\int_0^{2\pi} f_1(\theta)e^{-\mathrm{i}n\theta}\mathrm{d}\theta=\hat{f}_1(n).
	\]
	It satisfies
	\[
	f(e^{\mathrm{i}\theta})=u(\bm{x})|_{\partial O_1},
	\]
	and since
	\[
	\varlimsup_{|n|\rightarrow\infty}\sqrt[|n|]{|c_n|}=\frac{1}{\rho},
	\]
	$f(z)$ is analytic in $\{z\in\mathbb{C}|  1/\rho <|z|<\rho\}$.

\end{appendices}
\end{document}